\documentclass[11pt]{article}

\usepackage[a4paper]{geometry}
\usepackage{amsmath}
\usepackage{amsfonts}
\usepackage{setspace}
\usepackage{amssymb}
\usepackage{enumitem}
\usepackage{bbold} 
\usepackage{comment}
\usepackage{pgf,tikz}
\usetikzlibrary{arrows}
\usepackage[justification=centering, font=small]{caption}
\usepackage[unicode=true,pdfusetitle, bookmarks=true,bookmarksnumbered=true,bookmarksopen=true,
 breaklinks=false,pdfborder={0 0 1},backref=false,colorlinks=false]
 {hyperref}
\makeatletter
\g@addto@macro\@floatboxreset\centering
\makeatother

\usepackage{amsthm}
\newtheorem{thm}{Theorem} 
\newtheorem{lem}[thm]{Lemma}
\newtheorem{cor}[thm]{Corollary} 
\newtheorem{defn}[thm]{Definition}
\newtheorem{claim}[thm]{Claim}
\newtheorem{prop}[thm]{Proposition}
\newtheorem{ques}[thm]{Question}

\newtheorem{obs}[thm]{Observation}
\makeatletter
\newenvironment{boldproof}[1][\proofname] {\par\pushQED{\qed}\normalfont\topsep6\p@\@plus6\p@\relax\trivlist\item[\hskip\labelsep\bfseries#1\@addpunct{.}]\ignorespaces}{\popQED\endtrivlist\@endpefalse}
\makeatother

\newcommand{\EE}{\mathcal{E}}

\newcommand{\h}{\mathcal{H}}
\newcommand{\R}{\mathcal{R}}
\newcommand{\LL}{\mathcal{L}}

\newcommand{\abs}[1]{\left\lvert{#1}\right\rvert}
\newcommand{\floor}[1]{\left\lfloor{#1}\right\rfloor}

\DeclareMathOperator{\ex}{ex}
\DeclareMathOperator{\thres}{th}

\begin{document}

\title{Uniformity thresholds for the asymptotic size of extremal Berge-$F$-free hypergraphs}
\author{D\'aniel Gr\'osz\thanks{Department of Mathematics, University of Pisa. e-mail: \mbox{groszdanielpub@gmail.com}} \and Abhishek Methuku\thanks{Department of Mathematics, Central European University, Budapest,
Hungary. e-mail: \mbox{abhishekmethuku@gmail.com}} \and Casey Tompkins\thanks{Alfr\'ed R\'enyi Institute of Mathematics, Hungarian Academy of
Sciences. e-mail: \mbox{ctompkins496@gmail.com}}}
\date{}
\maketitle

\begin{abstract}
Let $F = (U,E)$ be a graph and $\h = (V,\EE)$ be a hypergraph. We say that $\h$ contains a Berge-$F$ if there exist injections $\psi:U\to V$ and $\varphi:E\to \EE$ such that for every $e=\{u,v\}\in E$, $\{\psi(u),\psi(v)\}\subset\varphi(e)$. Let $\ex_r(n,F)$ denote the maximum number of hyperedges in an $r$-uniform hypergraph on $n$ vertices which does not contain a Berge-$F$. 

For small enough $r$ and non-bipartite $F$, $\ex_r(n,F)=\Omega(n^2)$; we show that for sufficiently large $r$, $\ex_r(n,F)=o(n^2)$. Let $\thres(F) = \min\{r_0 :\ex_r(n,F) = o(n^2) \text{ for all } r \ge r_0 \}$. We show lower and upper bounds for $\thres(F)$, the uniformity threshold of $F$. In particular, we obtain that $\thres(\triangle) = 5$, improving a result of Gy\H ori \cite{Gyori_triangle}.

We also study the analogous problem for linear hypergraphs. Let $\ex^L_r(n,F)$ denote the maximum number of hyperedges in an $r$-uniform linear hypergraph on $n$ vertices which does not contain a Berge-$F$, and let the linear unformity threshold $\thres^L(F) = \min\{r_0 :\ex^L_r(n,F) = o(n^2) \text{ for all } r \ge r_0 \}$. We show that $\thres^L(F)$ is equal to the chromatic number of $F$. 
\end{abstract}

\section{Introduction and main results}

Let $F = (U,E)$ be a graph and $\h = (V,\EE)$ be a hypergraph. Generalizing the earlier definitions of Berge-path and Berge-cycle, Gerbner and Palmer \cite{Gerbner_Palmer} introduced the notion of Berge-$F$ hypergraphs.
\begin{defn}\label{Berge}
We say that $\h$ is a Berge-$F$ if there exist bijections $\psi:U\to V$ and $\varphi:E\to \EE$ such that for every $e=\{u,v\}\in E$, $\{\psi(u),\psi(v)\}\subset\varphi(e)$.
\end{defn}
Given a hypergraph $\h$, we denote by $\Gamma(\h)$ the $2$-shadow of $\h$, that is, the graph on the same vertex set, containing all $2$-element subsets of hyperedges from $\h$ as edges. Observe that $\h$ contains a Berge-$F$ as a subgraph if and only if $\Gamma(\h)$ contains a copy of $F$ such that $\h$ has a distinct hyperedge containing each edge of this copy of $F$.

For a graph $F$, let $\ex_r(n,F)$ denote the maximum number of hyperedges in an $r$-uniform hypergraph on $n$ vertices which does not contain a Berge-$F$ as a subgraph. The case when $r=2$ is the classical Tur\'an function $\ex(n,F)$. We will also consider what happens if we impose the additional assumption that the hypergraph is linear (i.e., any two hyperedges intersect in at most one element). We denote the maximum number of hyperedges in a linear $r$-uniform hypergraph on $n$ vertices which does not contain a Berge-$F$ by $\ex^L_r(n,F)$.

Gy\H ori, Katona and Lemons \cite{GyKaLe} generalized the Erd\H os--Gallai theorem to Berge-paths. Gy\H ori and Lemons \cite{Gyori_Lemons} gave upper bounds, and in some cases constructions, for $\ex_3(n,C_{2k+1})$. Gerbner and Palmer \cite{Gerbner_Palmer} gave bounds on $\ex_r(n,F)$ for $K_{s,t}$ and specifically $C_4$. It follows from Gy\H ori's results in~\cite{Gyori_triangle} that $\ex_r(n,\triangle)\le \frac{n^2}{8(r-2)}$ if $n$ is large enough. For $r=3,4$ this result is asymptotically sharp. We studied this problem in higher uniformities, and determined that, in fact, $\ex_r(n,\triangle) = o(n^2)$ when $r\ge5$, improving Gy\H ori's result. This will be obtained as a special case of more general theorems presented later.  

The following result can be proved easily.
\begin{prop}[Gerbner and Palmer \cite{Gerbner_Palmer}]
\label{easybound}
For any graph $F$ and $r \ge \abs{V(F)}$, we have
\begin{displaymath}
\ex_r(n,F) \le \ex(n,F) = O(n^2).
\end{displaymath}
\end{prop}

We include its proof in Section~\ref{ramseyboundproof}. By the Erd\H os--Stone theorem, for any bipartite graph $F$, we have $\ex(n,F) = o(n^2)$. Moreover, in the graph case, if $F$ is not bipartite, then we have $\ex(n,F) = \Omega(n^2)$. We will show that for any graph $F$ and for any sufficiently large $r$, we have $\ex_r(n,F) = o(n^2)$. We introduce the following threshold functions.

\begin{defn}
\label{thres}
Let $F$ be a graph. We define the uniformity threshold of $F$ as
\begin{displaymath}
\thres(F) = \min\{r_0\ge 2 :\ex_r(n,F) = o(n^2) \text{ for all } r \ge r_0 \}.
\end{displaymath}
We define the linear uniformity threshold of $F$ as
\begin{displaymath}
\thres^L(F) = \min\{r_0\ge 2 :\ex^L_r(n,F) = o(n^2) \text{ for all } r \ge r_0 \}.
\end{displaymath}
\end{defn}
Our first theorem gives an upper bound for the value of $\thres(F)$ for any graph $F$. (Note that if $F$ is bipartite, then $\thres(F)\le\abs{V(F)}$ by Proposition~\ref{easybound}.) The Ramsey number of two graphs $R(G,H)$ is defined to be the smallest $n$ such that every $2$-coloring of the edges of the complete graph $K_n$ contains a copy of $G$ in the first color or $H$ in the second color. For every $G$ and $H$ this number is known to be finite by Ramsey's theorem.

For a graph $F$ containing an edge $e$, let $F\setminus e$ denote the graph formed by deleting $e$ from $F$.
\begin{thm}
\label{ramseybound}
For any graph $F$ (with at least two edges), and any of its edges $e$, we have
\begin{displaymath}
\thres(F) \le R(F,F\setminus e).
\end{displaymath}
\end{thm}

Our next result is a construction giving a lower bound on $\thres(F)$. 
\begin{thm}
\label{construction}
Let $F$ be a graph with clique number $\omega(F)\ge 2$. For any $2\le r\le(\omega(F)-1)^2$, there exists an $r$-uniform, Berge-$F$-free hypergraph on $n$ vertices with $\Omega(n^2)$ hyperedges. Therefore we have
\begin{displaymath}
\thres(F) \ge (\omega(F)-1)^2+1.
\end{displaymath}
\end{thm} 

The above two theorems imply $\thres(\triangle)=5$.

Erd\H os, Frankl and R\"odl \cite{Erdos_Frankl_Rodl} constructed a linear $r$-uniform Berge-triangle-free hypergraph with more than $n^{2-\varepsilon}$ hyperedges for any $r\ge 3$ and $\varepsilon>0$. This implies that when $F=\triangle$, in our definition of the functions $\thres$ and $\thres^L$, $o(n^2)$ cannot be replaced by a function of $n$ with smaller exponent.

Finally, we consider linear hypergraphs. (In Section~\ref{constructionproof} we prove Theorem~\ref{construction} by blowing up a linear, Berge-$F$-free hypergraph.) It is easy to see that a linear hypergraph on $n$ vertices has at most $\binom{n}{2}$ hyperedges: fix a pair of vertices in each hyperedge; by the definition of a linear hypergraph, all these pairs must be distinct. Timmons \cite{Timmons} showed that (with our notation) $\thres^L(F) \le \abs{V(F)}$. We prove the following exact result.

\begin{thm}
\label{exact}
For any (non-empty) graph $F$, we have
\begin{displaymath}
\thres^L(F) = \chi(F),
\end{displaymath}
and for any $2\le r<\chi(F)$, there exists an $r$-uniform, linear, Berge-$F$-free hypergraph on $n$ vertices with $\Omega(n^2)$ hyperedges.
\end{thm}

Note that $\chi(F)$ may be bigger than the lower bound in Theorem~\ref{construction}, and it obviously also bounds $\thres(F)$ from below. Generalizing the proof of Theorem~\ref{construction}, we prove the following common generalization of Theorem~\ref{construction} and the lower bound on $\thres(F)$ coming from Theorem~\ref{exact}.


For a graph $F$, we define a $t$-admissible partition of $F$ as a partition of $V(F)$ into sets of size at most $t$, such that between any two sets there is at most one edge in $F$. `Contracting' a set $S$ of vertices in a graph produces a new graph in which all the vertices of $S$ are replaced with a single vertex $s$ such that $s$ is adjacent to all the vertices to which any of the vertices of $S$ was originally adjacent. 

\begin{thm}\label{general}
Let $F$ be a graph, and let $1\le t \le \abs{V(F)}-1$. Consider all the graphs obtained by contracting each set in some $t$-admissible partition of $F$ to a point, and let $c$ be the minimum of the chromatic numbers of all such graphs. If $c\ge 3$, then $\thres(F) \geq (c-1)t+1$.
\end{thm}
For $t=1$, the only $t$-admissible partition of a graph $F$ is putting every vertex into a different set, so $c=\chi(F)$, and we just get back the lower bound in Theorem~\ref{exact}. We also get Theorem~\ref{construction} as a special case of Theorem~\ref{general} when $t=\omega(F)-1$: In any $(\omega(F)-1)$-admissible partition of $F$, every vertex in a maximal clique of $F$ must belong to a different set of the partition. Indeed, no set of the partition may contain all vertices of an $\omega(F)$-clique.
But if a set $A$ from the partition contained two or more vertices of an $\omega(F)$-clique, and another set $B$ contained another vertex of that clique, then there would be two or more edges between $A$ and $B$, contradicting the definition of a $t$-admissible partition. This means that the graph we get after contracting all the sets of an $(\omega(F)-1)$-admissible partition contains an $\omega(F)$-clique, so its chromatic number is at least $\omega(F)$. Therefore $c \ge \omega(F)$.

As an example where Theorem~\ref{general} gives an improvement, consider $F =  K_{2,1,1}$. Putting $t=3$ and $c=3$, we get $\thres(K_{2,1,1})\ge 7$. Indeed, the only 3-admissible partition of $K_{2,1,1}$ is to put every vertex of $K_{2,1,1}$ into a different set. Theorem~\ref{construction} gives a lower bound of just 5, while Theorem~\ref{exact} gives 3. We give further corollaries of Theorem~\ref{general} about blowups of graphs in Section~\ref{constructionproof}.

Until now we focused on uniformities $r$ for which $\ex_r(n,F)$ is subquadratic. In Section~\ref{superquadratic}, we discuss the behavior of $\ex_r(n,F)$ as $r$ grows, more generally. In particular, we discuss uniformities $r$ for which $\ex_r(n,F)$ is superquadratic using the relationship between $\ex_r(n,F)$ and the maximum number of $K_r$'s in an $F$-free graph.

\section{Behavior of \texorpdfstring{$\ex_r(n,F)$}{ex\_r(n,F)} as $r$ increases}\label{superquadratic}

Alon and Shikhelman \cite{Alon_S} studied the maximum number of copies of a graph $T$ in an $F$-free graph on $n$ vertices, denoted $\ex(n,T,F)$. For example, in the following proposition, we paraphrase Propositions 2.1 and 2.2 in \cite{Alon_S}.
\begin{prop}[Alon and Shikhelman \cite{Alon_S}]\label{Alon_Shikhelman}
Let $r\ge2$. Then $\ex(n,K_r,F)=\Omega(n^r)$ if and only if $r<\chi(F)$. Moreover, if $r<\chi(F)$, then $\ex(n,K_r,F)=(1+o(1))\binom{\chi(F)-1}{r}\bigl(\frac{n}{\chi(F)-1}\bigr)^r$, otherwise $\ex(n,K_r,F)\le n^{r-\epsilon(r,F)}$ for some $\epsilon(r,F)>0$.
\end{prop}

For the $r<\chi(F)$ case, a construction showing $\ex(n,K_r,F)=\Omega(n^r)$ is a complete $r$-partite graph on $n$ vertices with roughly $\left\lfloor\frac{n}{r}\right\rfloor$ vertices in each part. It has chromatic number $r$, so it does not contain $F$, and it contains $\Omega(n^r)$ copies of $K_r$.

Clearly $\ex(n,K_r,F)\le\ex_r(n,F)$: Take an $F$-free graph with $\ex(n,K_r,F)$ $r$-cliques, and replace each $r$-clique with a hyperedge containing the vertices of the clique. The resulting hypergraph cannot contain a Berge-$F$, as its 2-shadow does not even contain a copy of $F$. The converse is not true. If we take a Berge-$F$-free hypergraph with $\ex_r(n,F)$ hyperedges, and we replace its hyperedges with $r$-cliques (i.e., we take its 2-shadow), it might contain a copy of $F$. The upper bound in the following proposition, which relates $\ex_r(n,F)$ to $\ex(n,K_r,F)$, was discovered by Gerbner and Palmer \cite{G_P2}. As the proof is very simple, we include it for completeness.

\begin{prop}[Gerbner and Palmer \cite{G_P2}]\label{relation}
For any $r\ge3$, \[\ex(n,K_r,F)\le\ex_r(n,F)\le\ex(n,K_r,F)+\ex(n,F).\]
\end{prop}
\begin{proof}
We have already seen $\ex(n,K_r,F)\le\ex_r(n,F)$. To prove $\ex_r(n,F)\le\ex(n,K_r,F)+\ex(n,F)$, let $\h$ be an $r$-uniform, Berge-$F$-free hypergraph on a vertex set $V$ of $n$ elements. We consider the hyperedges of $\h$ one-by-one, and we will mark elements of $\binom{V}{2}\cup\binom{V}{r}$. For each hyperedge, we mark a pair of its vertices that we have not marked yet; if all those pairs are already marked, then we mark the hyperedge itself.

Let $\tilde\h$ be the set of the marked pairs and hyperedges. $\tilde\h\cap\binom{V}{2}$ is a graph with no copy of $F$. Indeed, since we only marked one edge for each hyperedge, if the graph contained a copy of $F$, its edges would be contained by distinct hyperedges of $\h$, which would form a Berge-$F$. So $\bigl|\tilde\h\cap\binom{V}{2}\bigr|\le\ex(n,F)$. Meanwhile each hyperedge in $\tilde\h\cap\binom{V}{r}$ was marked because each pair of vertices in it had already been marked, so they form an $r$-clique in $\tilde\h\cap\binom{V}{2}$. But $\tilde\h\cap\binom{V}{2}$ is $F$-free, so the number of $r$-cliques in it is at most $\ex(n,K_r,F)$. Thus, $\bigl|\tilde\h\cap\binom{V}{r}\bigr|\le\ex(n,K_r,F)$. Since $\abs{\h}=\bigl|\tilde\h\bigr|$, the proof is complete.
\end{proof}


Proposition \ref{relation} implies that the two functions $\ex(n,K_r,F)$ and $\ex_r(n,F)$ differ by only $O(n^2)$. So $\ex_r(n,F)=O(n^2)$ if and only if $\ex(n,K_r,F)=O(n^2)$, and we have $\ex_r(n,F)=\omega(n^2)$ if and only if $\ex(n,K_r,F)=\omega(n^2)$. Moreover, if $\ex(n,K_r,F)=\omega(n^2)$, then $\ex_r(n,F)=(1+o(1))\ex(n,K_r,F)$. If $F$ is bipartite, since $\ex(n,F) = o(n^2)$, the difference is even smaller --- only $o(n^2)$. So for bipartite $F$, $\ex_r(n,F)=o(n^2)$ if and only if $\ex(n,K_r,F)=o(n^2)$, and if $\ex(n,K_r,F)=\Omega(n^2)$, then $\ex_r(n,F)=(1+o(1))\ex(n,K_r,F)$.

On the other hand, note that for any non-bipartite $F$, even if we know $\ex(n,K_r,F) \allowbreak= o(n^2)$, Proposition \ref{relation} does not imply $\ex_r(n,F)= o(n^2)$; so $\ex(n,K_r,F)$ does not tell us much about $\thres(F)$. 

Combining Proposition \ref{Alon_Shikhelman} and Proposition \ref{relation}, we can obtain the following nice proposition discovered by Palmer et al.\ \cite{Palmer}. We note, however, that the proof given in \cite{Palmer} is different from the simple proof mentioned here.

\begin{prop}
\label{chromatic_number_threshold}
Let $r\ge2$. If $r<\chi(F)$, then $\ex_r(n,F)=\Theta(n^r)$ and if $r\ge\chi(F)$, then $\ex_r(n,F)=o(n^r)$.

More precisely, if $r<\chi(F)$, then $\ex_r(n,F) = (1+o(1))\binom{\chi(F)-1}{r}\bigl(\frac{n}{\chi(F)-1}\bigr)^r$, and if $r\ge\chi(F)$, then $\ex_r(n,F)\le n^{r-\epsilon(r,F)}$ for some $\epsilon(r,F)>0$.
\end{prop}

Below we outline some interesting facts about the behavior of $\ex_r(n,F)$ as $r$ grows.

Proposition \ref{chromatic_number_threshold} shows that as $r$ increases from $2$ until $\chi(F)-1$, the function $\ex_r(n,F)$ increases, and from $r = \chi(F)$, it is $o(n^r)$. From $r \ge \abs{F}$ (at the latest), it becomes $O(n^2)$ again (by Proposition \ref{easybound}). However, the decrease does not stop there. As shown by Theorem \ref{ramseybound}, from some point, it becomes $o(n^2)$.

In general, we do not know much about the behavior of $\ex_r(n, F)$ when $r$ is between $\chi(F)$ and $\abs{F}-1$. In the special case of $F=K_s$, we know more. As $r$ increases from $\chi(F)-1$ to $\chi(F)$, $\ex_r(n, F)$ immediately jumps from $\Theta(n^{\chi(F)-1})$ to $O(n^2)$ (by Proposition~\ref{easybound} since  $\abs{V(F)}=\chi(F) =s$), and it is at most $O(n^2)$ for all $r \ge \chi(F)$. It would be very interesting to determine the precise threshold $\thres(K_s)$ for when it becomes sub-quadratic.

It is also notable that $\ex_r(n,F)$ may increase with $r$ in the range  $\chi(F) \le r \le \abs{F}-1$ (as will be shown by the proposition below). Theorem~1.2 in Alon and Shikhelman's paper \cite{Alon_S} shows that if $s\ge 2r-2$ and $t \ge (s-1)! + 1$, then $\ex(n, K_r, K_{s,t}) = \Theta\bigl(n^{r-\binom{r}{2}/s}\bigr)$. It is easy to check that $n^{r-\binom{r}{2}/s}$ monotonously increases in $r$ between $2$ and $\floor{\frac{s}{2}}+1$, and $n^{r-\binom{r}{2}/s}=\Omega(n^{2.25})$ when $3\le r\le \floor{\frac{s}{2}}+1$. So combining this with Proposition \ref{relation}, we get
\begin{prop}
\label{Kst}
For any $2 \le r \le \floor{\frac{s}{2}}+1$ and $t\ge(s-1)!+1$, we have $\ex_r(n, K_{s,t}) = \Theta\bigl(n^{r-\binom{r}{2}/s}\bigr)$.
\end{prop}
Of course $\ex_2(n,K_{s,t}) = o(n^2)$, while $\ex_3(n,K_{s,t})=\Omega(n^{2.25})$ if $s\ge4$, which implies that $\thres(F)$ is not necessarily the smallest $r\ge 2$ for which $\ex_r(n,F)=o(n^2)$. (Then from some point later on --- $r \ge s+t$ at the latest --- it becomes sub-quadratic again and remains so.)

For a non-bipartite graph $F$, Theorem \ref{ramseybound} implies that for any $r \ge R(F, F \setminus e)$, we have $\ex_r(n,F) = o(\ex(n,F))$. However, it is unclear if the same holds for some bipartite graphs. If $F$ is a forest, then it is known that $\ex_r(n,F) = \Theta(n) = \Theta(\ex(n,F))$. 

\begin{ques}
Is there a bipartite graph $F$ containing a cycle for which the following statement holds: There exists an integer $r_0(F)$ such that $\ex_r(n,F) = o(\ex(n,F))$ for all $r \ge r_0(F)$? If yes, is the same statement true for every bipartite $F$ containing a cycle?
\end{ques}

The analogous question for linear hypergraphs was asked by Verstra\"ete \cite{Timmons}. For $F = C_4$, we ask the following, more precise question about the threshold.

\begin{ques}
Is it true that $\ex_r(n,C_4) = o(n^{1.5})$ for all $r \ge 7$?
\end{ques}
One can show that for $2 \le r \le 6$, we have $\ex_r(n,C_4) = \Omega(n^{1.5})$. Consider a bipartite $C_4$-free graph $G$ having $\Omega(n^{1.5})$ edges with parts $A$ and $B$. Let $1 \le i \le 3$ and $1 \le j \le 3$. Now replace each vertex $a \in A$ with $i$ vertices $a_1, \ldots, a_i$, and each vertex $b \in B$ with $j$ vertices $b_1, \ldots, b_j$, so that each edge $ab \in E(G)$ is replaced by the hyperedge $\{a_1, \ldots, a_i, b_1, \ldots, b_j\}$, Let $A'$ and $B'$ be the sets replacing $A$ and $B$ respectively. Clearly, the resulting hypergraph $H$ is $(i+j)$-uniform. 

We claim that $H$ is Berge-$C_4$-free. Indeed, suppose for a contradiction that $abcd$ is a $C_4$ in the $2$-shadow of $H$ such that $ab, bc, cd, da$ are contained in distinct hyperedges. Since $i, j \le 3$, it is impossible that the vertices of the $C_4$ are all contained in $A'$ or $B'$. Notice that two of the vertices $a,b,c,d$ must correspond to the same vertex of $G$, because otherwise $abcd$ would be a $C_4$ in $G$ as well. Furthermore, if two adjacent vertices of $abcd$ are in $A'$ (or in $B'$), then they correspond to the same vertex of $G$. We have the following cases.
\begin{itemize}
\item If two opposite vertices of $abcd$, say $a$ and $c$, are in $A'$, and $b$ and $d$ are in $B'$, then suppose w.l.o.g. that $a$ and $c$ correspond to the same vertex of $G$. Then $ab$ and $bc$ are not contained in distinct hyperedges of $H$, a contradiction.
\item If two adjacent vertices, say $a$ and $b$ are in $A'$, and $c$ and $d$ are in $B'$, then $a$ and $b$ correspond to the same vertex of $G$, and so do $c$ and $d$.  Then $ad$ and $bc$ are not contained in distinct hyperedges of $H$.
\item If three vertices, say $a,b,c$, are in $A'$, and $d$ is in $B'$ (or vice-versa), then $a,b,c$ correspond to the same vertex of $G$, so $ad$ and $cd$ are not contained in distinct hyperedges of $H$.
\end{itemize}
As $2 \le i+j \le 6$, this shows that $\ex_r(n,C_4) = \Omega(\ex(n,C_4)) = \Omega(n^{1.5})$ for $2 \le r \le 6$.



\section{Upper bound --- Proof of Theorem~\ref{ramseybound}}\label{ramseyboundproof}
First we prove Proposition~\ref{easybound} by showing that $\ex(n,F)$ is an upper bound for $\ex_r(n,F)$, whenever $r\ge\abs{V(F)}$.

\begin{boldproof}[Proof of Proposition~\ref{easybound}]
Assume $\h = (V,\EE)$ is an $r$-uniform hypergraph which contains no Berge-$F$. One by one, for every hyperedge $h \in \EE$ we take an edge $e \subset h$ which has not yet been taken. By our assumption $r \ge \abs{V(F)}$ we can always do this, for otherwise we would have a complete $K_r$, and the corresponding hyperedges would form a Berge-$F$. After completing this procedure, we obtain a graph $G$ in which the number of edges is equal to the number of hyperedges in $\h$. Clearly $G$ is $F$-free and thus has at most $\ex(n,F)$ edges, completing the proof.
\end{boldproof}
Another essential tool in some of our proofs is the graph removal lemma. We recall it here without proof.
\begin{lem}[Graph removal lemma]
Let $F$ be a fixed graph. For any $\varepsilon>0$, there is a $\delta>0$ such that for every graph $G$ which has at most $\delta\abs{V(G)}^{\abs{V(F)}}$ copies of $F$, there exists a set $S\subseteq E(G)$ of $\varepsilon n^2$ edges such that every copy of $F$ in $G$ contains at least one edge from $S$.
\end{lem}

We wish to apply the removal lemma to the 2-shadow of $F$, denoted $\Gamma(F)$. To this end, we prove the following claim.

\begin{claim}
\label{numberbound}
Let $F$ be a fixed graph, and let $e\in E(F)$ and $r\ge \abs{V(F)}$. Let $\h$ be an $r$-uniform hypergraph on $n$ vertices with no Berge-$F$. Then the number of copies of $F$ in $\Gamma(\h)$ is $o(n^{\abs{V(F)}})$.
\end{claim}
\begin{proof}
Any copy of $F$ in $\Gamma(\h)$ has at least two edges (and therefore at least three vertices) in some hyperedge of $\h$, otherwise the hyperedges containing the edges of $F$ would form a Berge-$F$. Thus we have the following upper bound: 
\begin{displaymath}
\# \{\text{$F$-copies in $\Gamma(\h)$}\} \le \abs{E(\h)} \binom{\abs{V(F)}}{3} n^{\abs{V(F)}-3}.
\end{displaymath}
$\binom{\abs{V(F)}}{3}$ is considered constant, and by Proposition~\ref{easybound} $\abs{E(\h)}=O(n^2)$. So the number of copies of $F$ is $O(n^{\abs{V(F)}-1})$ and so $o(n^{\abs{V(F)}})$.
\end{proof}

From now on, we consider $F$ to be a fixed graph, $e\in E(F)$, and $r\ge R(F,F\setminus e)$. We consider an $r$-uniform hypergraph $\h$ with no Berge-$F$.

By Claim~\ref{numberbound} and the graph removal lemma, there are $o(n^2)$ edges such that every copy of $F$ in the 2-shadow of $\h$ contains one of these edges. Call the set of these edges $\R$. 

\begin{claim}\label{counting}
Every hyperedge of $\h$ contains an edge from $\R$ which is contained in at most $\abs{E(F)}-1$ hyperedges.
\end{claim}
\begin{proof}
By contradiction, assume that there is a hyperedge $h$ such that every edge from $\R$ contained in $h$ is in at least $\abs{E(F)}$ hyperedges. By the definition of $\R$, $\Gamma(\{h\})\setminus \R$ cannot contain a copy of $F$. Applying Ramsey's theorem with the edges of $\Gamma(\{h\})\setminus \R$ colored with the first color and those in $\Gamma(\{h\})\cap\R$ colored with the second, we obtain that $\Gamma(\{h\})\cap\R$ must contain a copy of $F\setminus e$. Let $\hat e$ be an edge in $h$ whose addition would complete this copy of $F$. By our assumption we can select $\abs{E(F)}$ different hyperedges to represent every edge in this copy of $F$: $h$ itself for $\hat e$, and other hyperedges containing the rest of the edges. These hyperedges form a Berge-$F$ in $\h$, a contradiction.
\end{proof}
We are now ready to complete the proof of Theorem~\ref{ramseybound}. For every hyperedge $h \in \h$ we apply Claim~\ref{counting} to find an edge $e\in \R$, $e \subset h$ which is contained in at most $\abs{E(F)}-1$ hyperedges. It follows that the number of edges in $\h$ is bounded by $\abs{E(F)}-1$ times the number of edges found in this way, and thus
\begin{displaymath}
\abs{E(\h)} \le (\abs{E(F)}-1)\abs{\R} = o(n^2).
\end{displaymath}

\section{Linear hypergraphs --- Proof of Theorem~\ref{exact}}\label{exactproof}
\subsection{Construction showing \texorpdfstring{$\thres^L(F) \ge \chi(F)$}{th\textasciicircum L(F) >= chi(F)}}\label{linearconstruction}
First, we show that for any $F$ we have $\thres^L(F) \ge \chi(F)$. Let $2\le r\le\chi(F)-1$. We construct an $r$-uniform linear hypergraph on $n$ vertices with $\Omega(n^2)$ edges and no Berge-$F$. Take $r$ sets $V_1,V_2,\dots,V_{r}$ of $\bigl\lfloor\frac{n}{r}\bigr\rfloor$ vertices each. For each $i$, $1\le i \le r$, let $V_i = \{v_{i,1},v_{i,2},\dots,v_{i,\lfloor n/r \rfloor}\}$. The hyperedges are the sets of $v_{i,j}$'s of the form $e_{x,m}=\{v_{1,x},v_{2,x+m},\dots,v_{r,x+(r-1)m}\}$ where $x \in \bigl\{1,2,\dots,\bigl\lfloor\frac{n}{2r}\bigr\rfloor\bigr\}$ and $m \in \bigl\{0,1,\dots,\bigl\lfloor\frac{n}{2r(r-1)}\bigr\rfloor\bigr\}$.

\begin{figure}
\definecolor{aqaqaq}{rgb}{0.63,0.63,0.63}
\begin{tikzpicture}[line cap=round,line join=round,>=triangle 45,x=1.0cm,y=1.0cm]
\draw [line width=1.6pt] (0,0)-- (0,3);
\draw [line width=1.6pt] (1,0)-- (1,3);
\draw [line width=1.6pt] (2,0)-- (2,3);
\draw [line width=1.6pt] (3,0)-- (3,3);
\draw [line width=1.6pt] (4,0)-- (4,3);
\draw [line width=1.6pt,dash pattern=on 5pt off 5pt] (0,0)-- (3,3);
\draw [line width=1.6pt,dash pattern=on 5pt off 5pt] (2,0)-- (5,3);
\draw [line width=1.6pt,dotted] (0,0)-- (6,3);
\draw [line width=1.6pt,dash pattern=on 5pt off 5pt] (1,0)-- (4,3);
\draw [line width=1.6pt] (5,3)-- (5,0);
\draw [line width=1.6pt,dash pattern=on 5pt off 5pt] (3,0)-- (6,3);
\draw [line width=1.6pt,dash pattern=on 5pt off 5pt] (4,0)-- (7,3);
\draw [line width=1.6pt,dotted] (1,0)-- (7,3);
\draw [line width=1.6pt,dash pattern=on 5pt off 5pt] (5,0)-- (8,3);
\draw [line width=1.6pt,dotted] (2,0)-- (8,3);
\draw [line width=1.6pt,dotted] (3,0)-- (9,3);
\draw [line width=1.6pt,dotted] (4,0)-- (10,3);
\draw [line width=1.6pt,dotted] (5,0)-- (11,3);
\begin{scriptsize}
\fill [color=aqaqaq] (6,3) circle (1.5pt);
\fill [color=aqaqaq] (0,0) circle (1.5pt);
\fill [color=aqaqaq] (5,3) circle (1.5pt);
\fill [color=aqaqaq] (2,0) circle (1.5pt);
\fill [color=aqaqaq] (1,0) circle (1.5pt);
\fill [color=aqaqaq] (3,0) circle (1.5pt);
\fill [color=aqaqaq] (4,0) circle (1.5pt);
\fill [color=aqaqaq] (0,1) circle (1.5pt);
\fill [color=aqaqaq] (1,1) circle (1.5pt);
\fill [color=aqaqaq] (2,1) circle (1.5pt);
\fill [color=aqaqaq] (3,1) circle (1.5pt);
\fill [color=aqaqaq] (4,1) circle (1.5pt);
\fill [color=aqaqaq] (0,2) circle (1.5pt);
\fill [color=aqaqaq] (1,2) circle (1.5pt);
\fill [color=aqaqaq] (2,2) circle (1.5pt);
\fill [color=aqaqaq] (3,2) circle (1.5pt);
\fill [color=aqaqaq] (4,2) circle (1.5pt);
\fill [color=aqaqaq] (0,3) circle (1.5pt);
\fill [color=aqaqaq] (1,3) circle (1.5pt);
\fill [color=aqaqaq] (2,3) circle (1.5pt);
\fill [color=aqaqaq] (3,3) circle (1.5pt);
\fill [color=aqaqaq] (4,3) circle (1.5pt);
\fill [color=aqaqaq] (5,0) circle (1.5pt);
\fill [color=aqaqaq] (6,0) circle (1.5pt);
\fill [color=aqaqaq] (7,0) circle (1.5pt);
\fill [color=aqaqaq] (5,1) circle (1.5pt);
\fill [color=aqaqaq] (6,1) circle (1.5pt);
\fill [color=aqaqaq] (7,1) circle (1.5pt);
\fill [color=aqaqaq] (5,2) circle (1.5pt);
\fill [color=aqaqaq] (6,2) circle (1.5pt);
\fill [color=aqaqaq] (7,2) circle (1.5pt);
\fill [color=aqaqaq] (7,3) circle (1.5pt);
\fill [color=aqaqaq] (8,0) circle (1.5pt);
\fill [color=aqaqaq] (8,1) circle (1.5pt);
\fill [color=aqaqaq] (8,2) circle (1.5pt);
\fill [color=aqaqaq] (8,3) circle (1.5pt);
\fill [color=aqaqaq] (9,0) circle (1.5pt);
\fill [color=aqaqaq] (9,1) circle (1.5pt);
\fill [color=aqaqaq] (9,2) circle (1.5pt);
\fill [color=aqaqaq] (9,3) circle (1.5pt);
\fill [color=aqaqaq] (10,0) circle (1.5pt);
\fill [color=aqaqaq] (10,1) circle (1.5pt);
\fill [color=aqaqaq] (10,2) circle (1.5pt);
\fill [color=aqaqaq] (10,3) circle (1.5pt);
\fill [color=aqaqaq] (11,0) circle (1.5pt);
\fill [color=aqaqaq] (11,1) circle (1.5pt);
\fill [color=aqaqaq] (11,2) circle (1.5pt);
\fill [color=aqaqaq] (11,3) circle (1.5pt);
\draw (6,3) circle (1.5pt);
\draw (0,0) circle (1.5pt);
\draw (5,3) circle (1.5pt);
\draw (2,0) circle (1.5pt);
\draw (1,0) circle (1.5pt);
\draw (3,0) circle (1.5pt);
\draw (4,0) circle (1.5pt);
\draw (0,1) circle (1.5pt);
\draw (1,1) circle (1.5pt);
\draw (2,1) circle (1.5pt);
\draw (3,1) circle (1.5pt);
\draw (4,1) circle (1.5pt);
\draw (0,2) circle (1.5pt);
\draw (1,2) circle (1.5pt);
\draw (2,2) circle (1.5pt);
\draw (3,2) circle (1.5pt);
\draw (4,2) circle (1.5pt);
\draw (0,3) circle (1.5pt);
\draw (1,3) circle (1.5pt);
\draw (2,3) circle (1.5pt);
\draw (3,3) circle (1.5pt);
\draw (4,3) circle (1.5pt);
\draw (5,0) circle (1.5pt);
\draw (6,0) circle (1.5pt);
\draw (7,0) circle (1.5pt);
\draw (5,1) circle (1.5pt);
\draw (6,1) circle (1.5pt);
\draw (7,1) circle (1.5pt);
\draw (5,2) circle (1.5pt);
\draw (6,2) circle (1.5pt);
\draw (7,2) circle (1.5pt);
\draw (7,3) circle (1.5pt);
\draw (8,0) circle (1.5pt);
\draw (8,1) circle (1.5pt);
\draw (8,2) circle (1.5pt);
\draw (8,3) circle (1.5pt);
\draw (9,0) circle (1.5pt);
\draw (9,1) circle (1.5pt);
\draw (9,2) circle (1.5pt);
\draw (9,3) circle (1.5pt);
\draw (10,0) circle (1.5pt);
\draw (10,1) circle (1.5pt);
\draw (10,2) circle (1.5pt);
\draw (10,3) circle (1.5pt);
\draw (11,0) circle (1.5pt);
\draw (11,1) circle (1.5pt);
\draw (11,2) circle (1.5pt);
\draw (11,3) circle (1.5pt);
\draw (0.34,-0.2) node {$v_{1,1}$};
\draw (2.34,-0.2) node {$v_{1,3}$};
\draw (1.34,-0.2) node {$v_{1,2}$};
\draw (0.34,0.8) node {$v_{2,1}$};
\draw (1.34,0.8) node {$v_{2,2}$};
\draw (2.34,0.8) node {$v_{2,3}$};
\end{scriptsize}
\end{tikzpicture}
\caption{The construction in Section~\ref{linearconstruction} (with $r=4$, $n=48$). The solid, dashed and dotted lines represent $e_{x,0}$'s, $e_{x,1}$'s and $e_{x,2}$'s respectively.}
\end{figure}
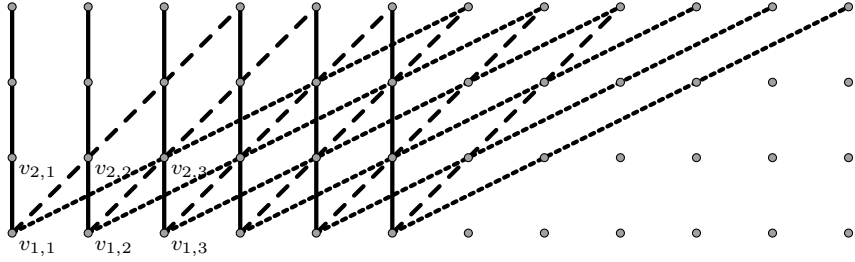

The number of hyperedges in this hypergraph is $\bigl\lfloor\frac{n}{2r}\bigr\rfloor \bigl(\bigl\lfloor\frac{n}{2r(r-1)}\bigr\rfloor+1\bigr)$. The hypergraph is linear: if two different vertices $v_{i_1,j_1}$ and $v_{i_2,j_2}$ are contained in a hyperedge, then $i_1\neq i_2$; and the two vertices uniquely determine the parameters of the hyperedge as $m=\frac{j_2-j_1}{i_2-i_1}$ and $x=j_1-(i_1-1)m$, so they cannot be contained in two hyperedges. The hypergraph has no Berge-$F$ since the 2-shadow contains no copy of $F$: a copy of $F$ would have to contain two adjacent vertices from the same $V_i$ (otherwise the $V_i$'s would form the classes of an $r$-coloring of $F$); but no hyperedge contains two vertices from the same $V_i$.

\subsection{Proof of sharpness: \texorpdfstring{$\thres^L(F) \le \chi(F)$}{th\textasciicircum L(F)<=chi(F)}}\label{linearsharpness}
Now we show that $\thres^L(F) \le \chi(F)$. Let $\h$ be an $r$-uniform ($r\ge\chi(F)$) linear hypergraph on $n$ vertices. A $w$-blowup of $K_{\chi(F)}$ is a complete $\chi(F)$-partite graph with $w$ vertices in each class.
\begin{lem}\label{blowup}
For large enough $w$ (which depends on $F$ and $r$, but not on $n$), if the 2-shadow of $\h$, denoted $\Gamma(\h)$, contains a $w$-blowup of $K_{\chi(F)}$, then $\h$ contains a Berge-$F$.
\end{lem}
\begin{proof}
Let $U$ be the vertices of a $w$-blowup of $K_{\chi(F)}$ in $\Gamma(\h)$, and let $U_1,\ldots,\allowbreak U_{\chi(F)}$ be its vertex classes. Let $v_1,\ldots,v_{\abs{V(F)}}$ be the vertices of $F$, and fix a proper coloring $c:V(F)\to\{1,\ldots,\chi(F)\}$. Consider a map $\psi:V(F)\to U$ such that $\forall i:\psi(v_i)\in U_{c(v_i)}$. For every edge $v_i v_j \in E(F)$, $c(v_i)\neq c(v_j)$, so $\psi(v_i)\psi(v_j)$ is an edge of $\Gamma(\h)$. If the hyperedges of $\h$ containing the edges $\psi(v_i)\psi(v_j)$ are distinct, then they form a Berge-$F$. 
We prove that $\h$ contains a Berge-$F$ by estimating the number of maps $\psi:V(F)\to U$ such that $\forall i:\psi(v_i)\in U_{c(v_i)}$, and upper bounding the number of such maps that do not yield a Berge-$F$.

There are more than $(w-\abs{V(F)})^{\abs{V(F)}}=\Omega(w^{\abs{V(F)}})$ such maps in total. Indeed, we can choose the image of $v_1,v_2,\ldots\in V(F)$ one after the other. For each vertex $v_i$, $\abs{U_{c(v_i)}}=w$, out of which at most $i-1$ vertices may already be taken, so we have more than $(w-\abs{V(F)}$ choices.

Now fix two edges $e_1,e_2\in E(F)$. We upper bound the number of maps such that the images of $e_1$ and $e_2$ are contained in the same hyperedge. There are $w^2$ ways to choose the images of the endpoints of $e_1=v_i v_j$ in $\Gamma(\h)$, since they have to be in $U_{c(v_i)}$ and $U_{c(v_j)}$ respectively. Because $\h$ is linear, there is only one hyperedge containing the image of $e_1$, so there are less than $\binom{r}{2}$ ways to choose the image of the endpoints of $e_2$ in $\Gamma(\h)$ so that it is contained in the same hyperedge as $e_1$. For the image of the remaining vertices $V(F)\setminus(e_1\cup e_2)$ we have less than $w^{\abs{V(F)}-3}$ or $w^{\abs{V(F)}-4}$ choices ($e_1\cup e_2$ contains 3 or 4 vertices depending on whether $e_1$ and $e_2$ share a vertex).

In total, considering all pairs $e_1,e_2\in F$, we have less than $\binom{\abs{E(F)}}{2}\binom{r}{2}w^2 w^{\abs{V(F)}-3}=O(w^{\abs{V(F)}-1})$ maps $\psi$ which do not yield a Berge-$F$. So for large enough $w$, $\h$ must contain a Berge-$F$.
\end{proof}

By Proposition~\ref{Alon_Shikhelman}, if a graph contains $\Omega(n^{\chi(F)})$ copies of $K_{\chi(F)}$, then (for large enough $n$) it contains any graph of chromatic number $\chi(F)$, including an arbitrarily large (constant) blowup of $K_{\chi(F)}$. Therefore, by Lemma~\ref{blowup}, assuming that $\h$ does not contain a Berge-$F$, $\Gamma(\h)$ contains only $o(n^{\chi(F)})$ copies of $K_{\chi(F)}$.

By the graph removal lemma, there is a set $S$ of $o(n^2)$ edges in $\Gamma(\h)$ such that every copy of $K_{\chi(F)}$ in $\Gamma(\h)$ contains at least one edge from $S$. Since $r\ge\chi(F)$, the 2-shadow of every hyperedge contains a $K_{\chi(F)}$, and therefore it contains an edge from $S$. Since $\h$ is a linear hypergraph, every edge in $\Gamma(\h)$ is contained by only one hyperedge, so $\abs{E(\h)}\le\abs{S}=o(n^2)$.

\section{Lower bound --- Proof of Theorems~\ref{construction} and~\ref{general}}\label{constructionproof}
\begin{defn}\label{defn:blowup}
Given a $k$-uniform hypergraph $\h$ on a vertex set $V$, a blowup of $\h$ by a factor of $w$ is a $kw$-uniform hypergraph $\h'$ obtained by replacing each vertex $u_i$ of $\h$ by $w$ vertices $v_{i,1},\ldots,v_{i,w}$; the hyperedges of the new hypergraph are $\left\{\{v_{i,j}:u_i \in e, j=1,\ldots,w\}:e\in\h\right\}$. We say that the vertices $v_{i,1},\ldots,v_{i,w}$ of $\h'$ \emph{originate from} the vertex $u_i$ of $\h$, and the hyperedge $\{v_{i,j}:u_i \in e, j=1,\ldots,w\}$ originates from the hyperedge $e$ of $\h$. We may also blow up the vertices of $\h$ by different factors, replacing a vertex $u_i\in V$ with $w(u_i)$ vertices ($w(u_i)\ge1$). If $\sum_{u\in e}w(u)=r$ for every $e\in\h$, then the new hypergraph is $r$-uniform.
\end{defn}
Note that this blowup definition is not analogous with the graph blowup definition used in Section~\ref{linearsharpness}. 

To motivate the reader, we first show the (simpler) proof of Theorem~\ref{construction}, and then generalize it to prove Theorem~\ref{general}. 

\begin{boldproof}[Proof of Theorem~\ref{construction}]Let $\omega(F)=s$. Assume for simplicity that $s-1$ divides $n$ (if it does not, then take the construction below on $(s-1)\bigl\lfloor\frac{n}{s-1}\bigr\rfloor$ vertices, and supplement it with a few isolated vertices). We construct an $(s-1)^2$-uniform hypergraph on $n$ vertices which does not contain a Berge-$K_s$. Since $K_s$ is a subgraph of $F$, it is easy to see that it does not contain a Berge-$F$ either. By Theorem~\ref{exact}, we have a linear $(s-1)$-uniform hypergraph $\LL$ on $\bigl\lfloor\frac{n}{s-1}\bigr\rfloor$ vertices with $\Omega(n^2)$ hyperedges that does not contain a Berge-$K_s$. Let $\h$ be the $(s-1)^2$-uniform hypergraph obtained by blowing up $\LL$ by a factor of $s-1$. $\h$ has the same number of hyperedges as $\LL$.

Assume by contradiction that $\h$ has a Berge-$K_s$. Then there is an $s$-clique in the $2$-shadow graph $\Gamma(\h)$. Let $v_1,\ldots,v_s$ be the vertices of an $s$-clique in $\Gamma(\h)$ which corresponds to a Berge-$K_s$ in $\h$. Let $u_i$ be the vertex of $\LL$ that $v_i$ originates from. Because the blow-up factor is $s-1$, it is impossible for all $u_i$'s to be the same vertex. It is also impossible for all $u_i$'s to be different, since then the Berge-$K_s$ in $\h$ would correspond to a Berge-$K_s$ in $\LL$. 
Thus we have $u_i=u_j\ne u_k$ for some $i\ne j\ne k$. But since $\LL$ is linear, there is at most one hyperedge in $\LL$ containing $u_i=u_j$ and $u_k$, so there are no distinct hyperedges in $\h$ containing the edges $v_i v_k$ and $v_j v_k$, contradicting that those vertices are part of a Berge-$K_s$ in $\h$.

Using the construction in Section~\ref{linearconstruction}, we can construct $r$-uniform hypergraphs for $r<(s-1)^2$ similarly. Let the sets $V_i$ be defined as in Section~\ref{linearconstruction}. If $r>s-1$, then blow up each vertex in $V_i$ by the same factor $w_i$, where $1 \le w_i \le s-1$, in such a way that  $\sum_i w_i = r$. If $r\le s-1$, then just take an $r$-uniform linear hypergraph with no Berge-$K_s$.
\end{boldproof}

\begin{boldproof}[Proof of Theorem~\ref{general}]
For any $r$ between 2 and $(c-1)t$, we construct an $r$-uniform hypergraph on $n$ vertices with $\Omega(n^2)$ hyperedges and no Berge-$F$. Since putting each vertex of $F$ in a separate set is a $t$-admissible partition, $c\le\chi(F)$. If $r<c$, just take an $r$-uniform linear hypergraph with no Berge-$F$ (given by Section~\ref{linearconstruction}). Otherwise let $\LL$ be the linear $(c-1)$-uniform hypergraph with $\Omega(n^2)$ hyperedges given by Section~\ref{linearconstruction} with $c$ in the place of $r$. $\LL$ does not contain any Berge-$G$ with $\chi(G)\ge c$. Now fix blow-up factors $w_1,\ldots,w_{c-1}$ between 1 and $t$ such that $\sum w_i = r$, and let $\h$ be a blow-up of $\LL$ obtained by blowing up every vertex in $V_i$ by $w_i$, for all $i$. $\h$ is $r$-uniform, and it has $\Omega(n^2)$ hyperedges.

Let $\abs{V(F)}=s$. Assume 
that $\h$ contains a Berge-$F$. Let $v_1,\ldots,v_s$ be the vertices of $F$, and let $\psi$ be the bijection that maps the vertices of $F$ to the vertices of the Berge-$F$ in $\h$ (as in Definition~\ref{Berge}). Let $\tilde\psi(v_i)$ be the vertex of $\LL$ from which $\psi(v_i)$ originates. Now partition the vertices of $F$ with $v_i$ and $v_j$ belonging to the same set if $\tilde\psi(v_i)=\tilde\psi(v_j)$. We claim that this is a $t$-admissible partition. Indeed, first notice that at most $t$ vertices of $\h$ originate from any vertex of $\LL$ because the blow-up factors were taken between 1 and $t$. So the size of each set in the partition is at most $t$. Now assume for a contradiction that there are two different edges of $F$, $v_iv_j$ and $v_kv_l$,  between some two sets of the partition. In other words, let $\tilde\psi(v_i)=\tilde\psi(v_l)$, and $\tilde\psi(v_j)=\tilde\psi(v_l)$. But $\LL$ is linear, so there is at most one hyperedge containing both $\tilde\psi(v_i)$ and $\tilde\psi(v_j)$, so there are no distinct hyperedges in $\h$ containing the edges $\psi(v_i)\psi(v_j)$ and $\psi(v_k)\psi(v_l)$ of $\Gamma(\h)$, contradicting the assumption that $\psi$ maps to a Berge-$F$ in $\h$.

So we have a $t$-admissible partition of $F$. Let $G$ be the graph obtained by contracting each set in the partition. $G$ is isomorphic to the graph $\tilde G$ with vertex set $\bigl\{\tilde\psi(v):v\in V(F)\bigr\}$ and edge set $\bigl\{\tilde\psi(v_i)\tilde\psi(v_j):v_iv_j\in E(F),\tilde\psi(v_i)\ne\tilde\psi(v_j)\bigr\}$. All the edges $\psi(v_i)\psi(v_j)$ of $\Gamma(\h)$ corresponding to edges $v_iv_j$ of $F$ are contained in distinct hyperedges of $\h$, since $\psi$ maps to a Berge-$F$. Since the hyperedges of $\h$ originate from distinct hyperedges of $\LL$, all edges $\tilde\psi(v_i)\tilde\psi(v_j)$ of $\tilde G$ are contained in distinct hyperedges of $\LL$, so $\LL$ contains a Berge-$G$. But $\chi(G)\ge c$, contradicting the assumption that $\LL$ does not contain any Berge-$G$ with $\chi(G)\ge c$.
\end{boldproof}

We show two corollaries of Theorem~\ref{general}. The following is a simple observation that helps in reasoning about admissible partitions (already alluded to after the statement of Theorem~\ref{general}).
\begin{obs}\label{admissibleobs}
If two vertices $v$ and $w$ are in the same set $A$ of a $t$-admissible partition of $F$, and a third vortex $u$ is connected to both of them, then $u\in A$: if it was in a different set $B$, then $uv$ and $uw$ would be two edges between $A$ and $B$, which is not allowed in a $t$-admissible partition.
\end{obs}

In the following corollaries, a blowup $F$ of $K_s$ is a blow up of $K_s$  in the usual graph sense, where each vertex $v_i \in V(K_s)$ may be blown up by a different factor $w_i$. 
\begin{cor}
Let $s\ge3$, and let $F$ be an arbitrary blowup of $K_s$. Then $\thres(F)\ge(s-1)(\abs{V(F)}-1)+1$.
\end{cor}
\begin{proof}
Let $t=\abs{V(F)}-1$. Let $V(F) = V_1\cup\ldots\cup V_s$, and let $V_i=\{v_{i,1},\ldots,v_{i,w(i)}\}$ where $v_{i,j}$ and $v_{k,l}$ are adjacent in $F$ if $i\ne k$.

We claim that the only $t$-admissible partition of $F$ is into singletons. Assuming this claim, there are no contractions to be made, so we have $c=\chi(F)=s$, proving the corollary.

Assume that $v_{i,j}$ and $v_{k,l}$ are in the same set $A$ in a $t$-admissible partition (where $i$ and $k$ may be different or equal). For every $p\notin\{i,k\}$ and $q\in{1,\ldots,w(p)}$, $v_{p,q}$ must be in $A$ by Observation~\ref{admissibleobs}, since $v_{p,q}$ is connected to both $v_{i,j}$ and $v_{k,l}$. Now $A$ contains at least one vertex from every $V_i$, and by choosing appropriate pairs of vertices, it is easy to see that the remaining vertices must be in $A$ too. But $t<\abs{V(F)}$, so putting all vertices in the same set is not a $t$-admissible partition.
\end{proof}

\begin{cor}
Let $G$ be a connected graph on the vertex set $u_1,\ldots,u_s$, and let $F$ be a sufficiently large blowup of $G$: Let $V_i=\{v_{i,1},\ldots,v_{i,w(i)}\}$ with $w(i)\ge i$ and $w(1)\ge2$, and let $V(F) = V_1\cup\ldots\cup V_s$, where $v_{i,j}$ and $v_{k,l}$ are connected if $u_iu_k\in E(G)$. Then $\thres(G)\ge(\chi(F)-1)(\abs{V(F)}-1)+1$.
\end{cor}
\begin{proof}
Let $t=\abs{V(F)}-1$. We claim that in any $t$-admissible partition of $F$, no two vertices in the same $V_i$ belong to the same set. First we prove the corollary assuming this claim is true. Given a $t$-admissible partition of $F$, we select vertices $v_{1,j_1},\ldots,v_{s,j_s}$ one-by-one such that they belong to different sets in the partition. In step $i$, we have at least $i$ choices of a vertex in $V_i$, out of which at most $i-1$ may belong to the partition that contains a vertex selected earlier, so we can choose the vertex $v_{i,j_i}$ greedily from the rest. The vertices $v_{1,j_1},\ldots,v_{s,j_s}$ induce a copy of $G$ in $F$. Since they belong to different sets, the graph obtained by contracting the sets will still contain a copy of $G$, so its chromatic number is $\chi(G)$. This holds for all $t$-admissible partitions of $F$, so $c=\chi(G)$, proving the corollary.

To prove the claim, assume that $v_{i,j}$ and $v_{i,l}$ are in the same set $A$ in a $t$-admissible partition of $F$. Let $u_p$ be a neighbor of $u_i$ in $G$. Then every $v_{p,q}$ is connected to both $v_{i,j}$ and $v_{i,l}$, so $V_p\subset A$ by Observation~\ref{admissibleobs}. Since $G$ is connected, and $\abs{V_i}=w(i)\ge2$ for every $i$, it follows that every $V_i$ is a subset of $A$. But then $\abs{A}=\abs{V(F)}>t$, contradicting that the partition is $t$-admissible.
\end{proof}

\section*{Acknowledgements}\phantomsection\addcontentsline{toc}{section}{Acknowledgements}
The research of the second and third author was partially supported by the National Research, Development and Innovation Office NKFIH under the grant K116769. 

\phantomsection

\end{document}